 \documentclass[11pt]{article}

\usepackage[dvipsnames]{xcolor}
\usepackage{amsmath, amsthm, mathtools}
\usepackage{amsfonts}
\usepackage{amssymb}
\usepackage{paralist}
\usepackage{extarrows}
\usepackage{framed}
\usepackage[colorlinks,linkcolor=blue, citecolor=Green]{hyperref}
\usepackage[left=2.50cm, right=2.50cm]{geometry}
\usepackage[numbers, sort]{natbib}
\newtheorem{theorem}{Theorem}[section]
\newtheorem{proposition}[theorem]{Proposition}
\newtheorem{definition}[theorem]{Definition}
\newtheorem{lemma}[theorem]{Lemma}

\newtheorem{remark}[theorem]{Remark}

\DeclareMathOperator{\diag}{diag}
\DeclareMathOperator{\Div}{div}
\newcommand{\R}{\mathbb{R}}

\DeclareMathOperator{\sgn}{sgn}
\numberwithin{equation}{section}

\allowdisplaybreaks

\title{\bf A strong-form stability for a class of $L^p$ Caffarelli-Kohn-Nirenberg  interpolation inequality\thanks{Supported by National Key R\&D Program of China (Grant 2023YFA1010001) and NSFC(12171265,12271184). E-mail address:
zhangyf22@mails.tsinghua.edu.cn (Zhang), zou-wm@mail.tsinghua.edu.cn (Zou)} }
\author{{\bf Yingfang Zhang, Wenming Zou}  \\ {\footnotesize \it  Department of Mathematical Sciences, Tsinghua University, Beijing 100084, China.} }
\date{}

\begin{document}
\maketitle

\begin{abstract}
\noindent We study the stability of a class of Caffarelli-Kohn-Nirenberg (CKN) interpolation inequality   and establish a strong-form stability as following:
    \begin{equation*}
        \inf_{v\in\mathcal{M}_{p,a,b}}\frac{ \|u-v\|_{H_b^p} \|u-v\|_{L^p_a}^{p-1} }{\|u\|_{H^p_b}\|u\|_{L^p_a}^{p-1}} \le C\delta_{p,a,b}(u)^{t},
    \end{equation*}
where $t=1$ for $p=2$ and $t=\frac{1}{p}$ for $p>2$, and $\delta_{p,a,b}(u)$ is deficit of the CKN. We also note that it is impossible to establish stability results for $\|\cdot\|_{H_b^p}$ or $\|\cdot\|_{L_a^p}$ separately. Moreover, we consider the second-order CKN inequalities and establish similar results for radial functions.

\vskip0.1in



\end{abstract}

\section{Introduction}
Recall the Caffarelli-Kohn-Nirenberg (CKN) inequality:
\begin{theorem}[\cite{CKN1984}]
    For $N\ge 1$, let $s,p,q,a,b,c$ and $\theta$ satisfy the following conditions:
    \begin{align*}
        & p,q\ge 1,\quad s>0,\quad 0\le \theta\le 1;\\
        & \frac 1p+\frac{b}{N}>0,\quad \frac 1q + \frac{a}{N}>0, \quad \frac 1s + \frac{c}{N}>0;\\
        & \frac 1s + \frac {c}{N} = \theta\left(\frac{1}{p} + \frac{b-1}{N}\right) + (1-\theta)\left(\frac 1q + \frac{a}{N}\right);\\
        & c \le \theta b + (1-\theta)a; \\
        & \frac 1s \le \frac{\theta}{p} + \frac{1-\theta}{q},\quad \text{if}\ \theta =0 \ \text{or}\  \theta=1 \ \text{or} \ \frac 1s + \frac{c}{N} = \frac 1p + \frac{b-1}{N} = \frac 1q + \frac{a}{N}.
    \end{align*}
    Then there exists some positive constant $S$ such that
    \begin{equation}\label{CKN111}
        \big\| |x|^b \nabla u\big\|_{L^p(\R^N)}^\theta\ \big\| |x|^a u\big\|_{L^q(\R^N)}^{1-\theta} \ge S\big\| |x|^c u\big\|_{L^s(\R^N)}
    \end{equation}
    holds for all $u\in C_c^1(\R^N)$.
\end{theorem}
The largest $S$ that satisfies \eqref{CKN111} is its best constant, and a function that takes the equality is called a minimizer.

\subsection{Brief history of the stability for  CKN inequalities}
We are concerned with the stability of it --- can the {\it deficit } of $u$:
\[ \delta(u) \coloneqq \||x|^b\nabla u\|_{L^p}^\theta\||x|^au\|_{L^q}^{1-\theta} - S\||x|^cu\|_{L^s}\]
control  some distance of $u$ from the set of minimizers? Our quest to address this question unfolds in two primary ways: firstly, we need to figure out the best constant and the structure of the set of minimizers. Subsequently, we embark on identifying and employing suitable methods to estimate the distance.

\vskip0.1in
The question on the stability of functional inequalities traced back to Br\'{e}zis and Lieb \cite{Brezis1985}, and the initial result of the Sobolev inequalities (for the case $p=2$) was given by Bianchi and Egnell in \cite{Bianchi1991}. The quantitative form of the $L^p$ Sobolev inequality with $p\neq 2$ was proved by \cite{Cianchi2009, Figalli2018}. For more detailed results about the Sobolev-type inequalities, see \cite{Figalli2022, Chen2013, Gazzola2010, Neumayer2020, Talenti1976, Figalli2013, Cianchi2006, Fusco2007}.

\vskip0.1in
 When $\theta=1$, in the special case,  only the term of derivative remains on the left, and this is known as weighted Sobolev inequality. Fortunately, some methods of studying the stability of the Sobolev inequality remain applicable in this case, and there are already many stability results.  Chou and Chu \cite{Chou1993} gave the   best constant and minimizers for $p=2,\, (b,c)\neq (0,0)$. Based on their work, Wei and Wu \cite{Wei2022} established some stability results in this case. For $p\neq 2$, the explicit best constant was given by Horiuchi in \cite{Toshio1997}. See Zhou and Zou \cite{Zhou2024} and the references therein for stability results in this case.

\vskip0.1in

When it comes to the case $\theta\neq 1$, the inclusion of the interpolation terms in the CKN inequality makes it difficult to study. When $(a,b,c)=(0,0,0)$, it becomes the renowned Gagliardo-Nirenberg-Sobolev (GNS) inequality. Since there are no singular terms, we can use some rearrangement and symmetrization tricks to get its best constants and minimizers. Del Pino and Dolbeault \cite{DelPino2002} gave the best constant and minimizers for $p=2$, and \cite{Cordero-Erausquin2004} gave another proof by using mass transportation. Carlen and Figalli \cite{Carlen2013} used dimension reduction method to derive a stability result for the GNS inequality in $\R^2$ by a result for Sobolev inequality in $\R^4$. Later, Nguyen \cite{Nguyen2019} and Seuffert \cite{Seuffert2016} generalized their work.

\vskip0.1in

The case $\theta\neq 1,\, (a,b,c)\neq (0,0,0)$ is perhaps the most challenging one. Until recently, stability results in this case had not been extensively explored. In \cite{McCurdy2021},  McCurdy and Venkatraman gave a stability result for $(p,q,s)=(2,2,2),\, (a,b,c) = (1,0,0)$. Specifically, there exist constants $C_1>0,\,C_2(N) >0$ such that for all $u\in W^{1,2}(\R^N)$ such that $|x|u\in L^2(\R^N)$, the following inequality holds:
\begin{align*}
    \|\nabla u\|_{L^2(\R^N)}^2\||x|u\|_{L^2(\R^N)}^2 - \frac{N^2}{4}\|u\|_{L^2(\R^N)}^4 \ge C_1\|u\|_{L^2(\R^N)}d_1(u,E)^2 + C_2(N)d_1(u,E)^4,
\end{align*}
where $E=\{ce^{-\alpha|x|^2}:c\in\R,\,\alpha>0\}$ is the set of minimizers, and $d_1(u,E)$ denotes the distance from $u$ to $E$ under $L^2$-norm. Fathi \cite{Fathi2021} gave a short proof with explicit constants $(C_1,C_2)=(\frac 14, \frac 1{16})$. Further advancements were made by Cazacu, Flynn, Lam and Lu \cite{Cazacu2024}, which established an identity and used it to obtain the optimal constants $(C_1,C_2) = (N,1)$. Do, Flynn, Lam and Lu \cite{Do2023} followed their methods and generalized their result to the following $L^p$-CKN inequality:
\begin{theorem}[{\cite[Corollary 1.2]{Do2023}}]
    Let $N\ge 1,\, p>1,\, b+1-a>0$ and $b\le \frac{N-p}{p}$. Then for all $u\in C_0^\infty(\R^N\backslash\{0\})$,
    \begin{equation}\label{CKN}
        \left(\int_{\R^N} \frac{|\nabla u|^p}{|x|^{pb}}\,dx\right)^{1/p}\left(\int_{\R^N} \frac{|u|^p}{|x|^{pa}}\,dx\right)^{\frac{p-1}{p}} \ge \frac{N-1-(p-1)a-b}{p} \int_{\R^N} \frac{|u|^p}{|x|^{(p-1)a+b+1}}\,dx.
    \end{equation}
\end{theorem}
They also obtain the following  weak stability results for some special $(p,a,b)$:
\begin{theorem}[{\cite[Theorem 1.9]{Do2023}}]\label{Do's Thm}
    Let $p\ge 2,\, 0\le b <\frac{N-p}{p},\, a<\frac{Nb}{N-p}$, and $(p-1)a+b+1=\frac{pbN}{N-p}$. There exists a universal constant $C(N,p,a,b)>0$ such that for all $u\in C_0^\infty(\R^N\backslash\{0\})$,
    \begin{equation}\label{weak stability}
        \begin{aligned}
            & \left(\int_{\R^N} \frac{|\nabla u|^p}{|x|^{pb}}\,dx\right)^{1/p}\left(\int_{\R^N} \frac{|u|^p}{|x|^{pa}}\,dx\right)^{\frac{p-1}{p}} - \frac{N-1-(p-1)a-b}{p} \int_{\R^N} \frac{|u|^p}{|x|^{(p-1)a+b+1}}\,dx\\
            \ge{}& C(N,p,a,b) \inf_{v\in \mathcal M_{p,a,b}} \int_{\R^N} \frac{|u-v|^p}{|x|^{(p-1)a+b+1}}\,dx,
        \end{aligned}
    \end{equation}
    where its set of minimizers is
    \[ \mathcal{M}_{p,a,b} =\left\{k\exp\left(-\frac{\lambda}{b+1-a}|x|^{b+1-a}\right):k\in \R,\,\lambda>0\right\}.\]
\end{theorem}
\begin{theorem}[{\cite[Theorem 1.8]{Do2023}}]\label{Do's Thm 2}
    Let $\frac{N-2}{2} < b \le N-2$ and $N(b-a+3) = 2(3b-a+3)$. There exists a universal constant $C(N,a,b)>0$ such that for all $u\in C_0^\infty(\R^N\backslash\{0\})$,
    \begin{equation}\label{weak stability 2}
        \begin{aligned}
            &\left(\int_{\R^N} \frac{|\nabla u|^2}{|x|^{2b}}\,dx\right)^{1/2}\left(\int_{\R^N} \frac{|u|^2}{|x|^{2a}}\,dx\right)^{1/2} - \frac{3b-a-N+3}{2} \int_{\R^N} \frac{|u|^2}{|x|^{a+b+1}}\,dx\\
            \ge{}& C(N,a,b) \inf_{v\in \mathcal M_{2,a,b}} \int_{\R^N} \frac{|u-v|^2}{|x|^{a+b+1}}\,dx,
        \end{aligned}
    \end{equation}
    where its set of minimizers is
    \[ M_{2,a,b} = \left\{ k|x|^{2b+2-N}\exp\left(-\frac{\lambda}{b+1-a}|x|^{b+1-a}\right):k\in \R,\,\lambda>0\right\}.\]
\end{theorem}

While it is natural to anticipate that the deficit $\delta(u)$, as defined above, can control one of the two norms on the left side of \eqref{CKN}, unfortunately, it was derived in \cite{McCurdy2021} that, for any nonnegative constants $(M_1,M_2)\neq (0,0)$, there exists $u\in W^{1,2}(\R^N)$ with $|x|u\in L^2(\R^N)$ and $u^*\in E$ such that
\[ \|\nabla u\|_{L^2(\R^N)}^2\||x|u\|_{L^2(\R^N)}^2 - \frac{N^2}{4}\|u\|_{L^2(\R^N)}^4 \le M_1 \|\nabla(u-u^*)\|_{L^2(\R^N)}^2 + M_2 \||x|(u-u^*)\|_{L^2(\R^N)}^2. \]

For the $L^p$-CKN inequality, we will prove that this property holds as well (see Proposition \ref{prop1} below), hence, it is impossible to establish stability results for these two norms separately.

Recently, Chen and Tang \cite{chen2024} considered the optimal constants and minimizers of the following second-order CKN inequalities:
\begin{theorem}[{\cite[Theorem 2.5]{chen2024}}]
    Let $N\ge 1$ and $p>1$. Assume $(a,b)$ satisfy one of the following conditions:
    \begin{enumerate}
        \item [(1)] $b+(p-1)a+1>0$, $b-a+1>0$ and $pb+(p-1)N<0$;
        \item [(2)] $b+(p-1)a+1<0$, $b-a+1<0$ and $pb+(p-1)N>0$;
        \item [(3)] $b+(p-1)a+1=0$,
    \end{enumerate}
    then there holds
    \begin{equation}
        \left(\int_{\R^N} \frac{|\Delta u|^p}{|x|^{pb}}\,dx\right)^{1/p}\left(\int_{\R^N} \frac{|\nabla u|^p}{|x|^{pa}}\,dx\right)^{\frac{p-1}{p}} \ge \frac{|b+(p-1)a+(p-1)(N-1)|}{p}\int_{\R^N} \frac{|\nabla u|^p}{|x|^{b+(p-1)a+1}}],dx.
    \end{equation}
    Moreover, the equality holds if and only if
    \begin{align*}
        & u(x) = \Lambda \int_{|x|}^\infty r^{1-N}\exp\left(-\frac{\lambda r^{b-a+1}}{b-a+1}\right)\,dr,\quad \text{for}\ a,b \ \text{satisfy (1)}; \\
        & u(x) = \Lambda \int_{|x|}^\infty r^{1-N}\exp\left(\frac{\lambda r^{b-a+1}}{b-a+1}\right)\,dr,\quad \text{for} \ a,b \ \text{satisfy (2)}.
    \end{align*}
\end{theorem}
Their proof relied on the following identity:
\begin{theorem}[{\cite[Theorem 2.3 and 2.4]{chen2024}}]
    Let $N\ge 1$, $p>1$, $b-a+1>0$ and set $\epsilon = \sgn(pb+(p-1)N)$. For $u\in C_c^\infty(\R^N\backslash\{0\})$, there holds
    \begin{equation}\label{identity-I of 2nd order}
        \begin{aligned}
            &  \left(\int_{\R^N} \frac{|\Delta u|^p}{|x|^{pb}}\,dx\right)^{1/p}\left(\int_{\R^N} \frac{|\nabla u|^p}{|x|^{pa}}\,dx\right)^{\frac{p-1}{p}} - \frac{|b+(p-1)a+(p-1)(N-1)|}{p}\int_{\R^N} \frac{|\nabla u|^p}{|x|^{b+(p-1)a+1}}],dx\\
             & = \frac 1p \int_{\R^N} \frac{1}{|x|^{pb}}\mathcal{R}_p\left(-\lambda^{\frac 1p}|x|^{b-a-1}x\cdot \nabla u, \lambda^{-\frac{p-1}{p}}\Delta u \right)\,dx \\
            & \quad +\epsilon [b+(p-1)a+1]\int_{\R^N} \frac{|\nabla u|^{p-2}[(x\cdot \nabla u)^2-|x|^2|\nabla u|^2]}{|x|^{b+(p-1)a+3}}\,dx,
        \end{aligned}
    \end{equation}
    where
    \[ \mathcal{R}_p(s,t) \coloneqq |t|^p + (p-1)|s|^p - p|s|^{p-2}st,\quad s,t \in \R,\]
    and $\lambda = \frac{\left(\int_{\R^N} |x|^{-pb}|\Delta u|^p\right)^{1/p}}{\left(\int_{\R^N}|x|^{-pa}|\nabla u|^p\right)^{1/p}}$.
\end{theorem}

\vskip0.2in

\subsection{Main results of the current paper}
We begin by some notations. Let $N\ge 1$, $p>1$. The space $H^p_b(\R^N)$ is the closure of $C_c^\infty(\R^N\backslash\{0\})$ under the norm
$$ \|u\|_{H^p_b(\R^N)} \coloneqq \left(\int_{\R^N} |x|^{-pb}|\nabla u|^p\,dx\right)^{1/p}, $$
and $L^p_a(\R^N)$ is the closure of $C_c^\infty(\R^N\backslash\{0\})$ under the norm
\[ \|u\|_{L^p_a(\R^N)}\coloneqq\left(\int_{\R^N} |x|^{-pa}|u|^p\,dx\right)^{1/p}.\]
The space $\mathcal{H}_{b,a}^{p,q}$ is the closure of $C_c^\infty(\R^N\backslash\{0\})$ under the norm
\[ \|u\|_{\mathcal{H}_{b,a}^{p,q}(\R^N)} \coloneqq \|\Delta u\|_{L_b^p(\R^N)} + \|u\|_{H_a^q(\R^N)}.\]
In the following sections, for the sake of simplifying notations, we will omit $\R^N$ if there are no confusions.

We define the deficit of a function $u\in H^p_b\cap L^p_a$ to be
  \begin{equation}\label{ckn001} \delta_{p,a,b}(u)=\frac{\|u\|_{H_b^p}\|u\|_{L_a^p}^{p-1}}{\|u\|_{L_c^p}^p}-S_{p,a,b},   \end{equation}
where $c=\frac{(p-1)a+b+1}{p}$ and $S_{p,a,b}$ is the best constant of the CKN inequality. Similarly, we define the deficit of a function $u\in \mathcal{H}_{b,a}^{p,p}$ to be
  \begin{equation}\label{ckn002}\sigma_{p,a,b}(u) \coloneqq \frac{\|\Delta u\|_{L_b^p}\|u\|_{H_a^p}^{p-1}}{\|u\|_{H_c^p}^p} - K_{p,a,b}, \end{equation}
where $K_{p,a,b}$ is the best constant of the second-order CKN inequality.

\vskip0.12in

The first part of our paper focus on the CKN inequalities.  Firstly, we will prove the following proposition, which means it is impossible to establish stability results for $\|\cdot\|_{H_b^p}$ or $\|\cdot\|_{L_a^p}$ separately.
\begin{proposition}\label{prop1}
    Assume $1<p<N$ and $b+1\neq a$, then there do not exist universal constant $C>0$ and $\alpha > 0$ such that
    \begin{align*}
        \left( \frac{\inf_{v\in \mathcal{M}_{p,a,b}}\|u-v\|_{H^p_b} }{\|u\|_{L^p_c}} \right)^\alpha \le C \delta_{p,a,b}(u)
    \end{align*}
    or
    \begin{align*}
        \left( \frac{\inf_{v\in \mathcal{M}_{p,a,b}}\|u-v\|_{L^p_a} }{\|u\|_{L^p_c}} \right)^\alpha \le C \delta_{p,a,b}(u)
    \end{align*}
    holds for all $u \in H^p_b \cap L^p_a$.
\end{proposition}
In light of this, our focus in this paper shifts towards providing a strong-form stability with respect to the interpolation terms. Specifically, we observe  very different
behaviors for the cases of  $p=2$ and $p>2.$
\begin{theorem}\label{main p=2}
    Let  $p=2$ and  $(a,b)$ satisfy one of the following conditions:
    \begin{enumerate}
        \item [(1)] $0\le b < \frac{N-2}{2}$, $a<\frac{Nb}{N-2}$ and $a+b+1=\frac{2bN}{N-2}$;
        \item [(2)] $\frac{N-2}{2}<b\le N-2$ and $N(b-a+3) = 2(3b-a+3)$.
    \end{enumerate}
    Then there exists a constant $C$ depending only on $N,a,b$, such that for all $u\in H_b^2\cap L^2_a$, it holds
    \begin{equation}\label{result}
        \inf_{v\in\mathcal{M}_{2,a,b}}\frac{ \|u-v\|_{H_b^2} \|u-v\|_{L^2_a} }{\|u\|_{H_b^2}\|u\|_{L^2_a}} \le C\delta_{2,a,b}(u).
    \end{equation}
    We additionally mention that it is sharp (See Remark \ref{rmk 3.1}).
\end{theorem}
For the case $p>2$, we get
\begin{theorem}\label{main}
    Let $p>2,\, 0\le b <\frac{N-p}{p},\, a<\frac{Nb}{N-p}$, and $(p-1)a+b+1=\frac{pbN}{N-p}$. Then there exists a constant $C$ depending only on $N,a,b,p$, such that for all $u\in H_b^p\cap L_a^p$, it holds
    \begin{equation}\label{general result}
        \inf_{v\in\mathcal{M}_{2,a,b}}\frac{ \|u-v\|_{H_b^p} \|u-v\|_{L^p_a}^{p-1} }{\|u\|_{H^p_b}\|u\|_{L^p_a}^{p-1}} \le C\delta_{p,a,b}(u)^{1/p}.
    \end{equation}
\end{theorem}

\vskip0.3in
To obtain such statements, we introduce a stronger deficit:
\begin{equation}\label{ckn003}\tilde{\delta}_{p,a,b}(u) = \frac{\|u\|_{H^p_b}^p + (p-1) \|u\|_{L^p_a}^p}{\|u\|_{L_c^p}^p} - pS_{p,a,b}.\end{equation}
It is not difficult to see that $\tilde{\delta}_{p,a,b}(u)\ge p\delta_{p,a,b}(u)\ge 0$ (cf. \eqref{ckn001}).
Interestingly,  we  are able to prove that this  stronger deficit $\tilde{\delta}_{p,a,b}(u)$ can control the distance of $u$ from $\mathcal{M}_{p,a,b}$ under both $H^p_b$- and $L^p_a$-norms:
\begin{theorem}\label{plus version}
    Suppose that $p,a,b$ satisfy one of the conditions of Theorem \ref{main p=2} or Theorem \ref{main}, then there exists a constant $C$ depending only on $p,a,b,N$ such that for any $u\in H^p_b \cap L^p_a$,
    \begin{equation}\label{sum 1}
        \inf_{v\in \mathcal{M}_{p,a,b}}\frac{\|u-v\|_{H_b^p}^p + (p-1)\|u-v\|_{L^p_a}^p}{\|u\|_{H^p_b}^p + (p-1)\|u\|_{L^p_a}^p} \le C\tilde{\delta}_{p,a,b}(u)^{1/p}
    \end{equation}
    if $p>2$;
    \begin{equation}\label{sum 2}
        \inf_{v\in \mathcal{M}_{2,a,b}}\frac{\|u-v\|_{H_b^2}^2 + \|u-v\|_{L^2_a}^2}{\|u\|_{H^2_b}^2 + \|u\|_{L^2_a}^2} \le C\tilde{\delta}_{2,a,b}(u)
    \end{equation}
    if $p=2$.
\end{theorem}

Next, we state results for the second-order CKN inequalities. With identity \eqref{identity-I of 2nd order}, we can establish the following stability results for radial functions $u\in H_{b,a}^{p,p}$ under some conditons for $N,p,a,b$:
\begin{theorem}\label{main I - 2nd order}
    Assume $N\ge 1$, $2\le p<N$, $1-N\le b<\frac{N(1-p)}{p}$, $b-a+1>0$ and $b+(p-1)a+1 = \frac{Npb}{N-p}+\frac{p^2(N-1)}{N-p}$. Then their exists a constant $C$ depending only on $N,p,a,b$, such that for all radial function $u\in H_{b,a}^{p,p}$, it holds
    \begin{equation}\label{2nd weak inequalities}
        \begin{aligned}
            &\|\Delta u\|_{L_b^p}\|u\|_{H_a^p}^{p-1} - \frac{-b-(p-1)a-(p-1)(N-1)}{p}\|u\|_{H_c^p}^p\\
             & \ge{} C(N,p,a,b) \inf_{v\in\mathcal{M}_{p,a,b}^2} \|u-v\|_{H_c^p}^p,
        \end{aligned}
    \end{equation}
    where $\mathcal{M}_{p,a,b}^2$ is the set of minimizers:
    \[ \mathcal{M}_{p,a,b}^2 \coloneqq \left\{v(x) = \Lambda \int_{|x|}^\infty r^{1-N}\exp\left(-\frac{\lambda r^{b-a+1}}{b-a+1}\right)\,dr:\Lambda\in\R,\, \lambda>0\right\}.\]
\end{theorem}

Using this weak stability result, we can then obtain a strong-form version:
\begin{theorem}\label{main II - 2nd order}
    Under the same conditions as in Theorem \ref{main I - 2nd order}, their exists a constant $C$ depending only on $N,p,a,b$, such that for all radial function $u\in H_{b,a}^{p,p}$, it holds
    \begin{equation}\label{2nd strong inequalities}
        \inf_{v\in\mathcal{M}_{p,a,b}^2}\frac{\|\Delta(u-v)\|_{L_b^p}\|u-v\|_{H_a^p}^{p-1}}{\|\Delta u\|_{L_b^p}\|u\|_{L^p_a}^{p-1}} \le C \sigma_{p,a,b}(u)^{1/p}
    \end{equation}
    if $p>2$;
    \begin{equation}\label{2nd strong inequalities p=2}
        \inf_{v\in\mathcal{M}_{2,a,b}^2}\frac{\|\Delta(u-v)\|_{L_b^2}\|u-v\|_{H_a^2}}{\|\Delta u\|_{L_b^2}\|u\|_{L^2_a}} \le C \sigma_{2,a,b}(u)
    \end{equation}
    if $p=2$.
\end{theorem}

\vskip0.2in

The paper is organized as follows: In Section 2, we provide necessary preliminaries to establish our work. Especially, we give a statement for the variational problem $\inf_{v\in\mathcal{M}_{p,a,b}}\|u-v\|_{L^p_c}$. The proof of Proposition \ref{prop1} and Theorem \ref{plus version} are presented in Section 3. We will then use them to derive Theorem \ref{main p=2} and \ref{main}. The discussions for the second-order CKN inequalities are included in Section 4.

\section{Preliminaries}\label{section 2}

In the following, the notation $A\lesssim B$ (resp. $A\gtrsim B$) means there exists a constant $C$ depending at most on $N,p,a,b$ such that $A\le CB$ (resp. $A\ge CB$). We say $A\sim B$ if $A\lesssim B$ and $A\gtrsim B$.

\subsection{Some  observations for the CKN inequalities}
\begin{itemize}
    \item When $(p,a,b)$ satisfy conditions in Theorem \ref{main} or conditions (1) in Theorem \ref{main p=2}, then $S_{p,a,b}= \frac{N-1-(p-1)a-b}{p}$. Functions in $\mathcal{M}_{p,a,b}$ can be written as
    \[v(k,\lambda) = kC_1\lambda^{\alpha}\exp\left(-\frac{\lambda}{b+1-a}|x|^{b+1-a}\right),\]
    where
    \[ \alpha = \frac{N-1-(p-1)a-b}{(b+1-a)p} = \frac{S_{p,a,b}}{b+1-a},\]
    and
    \begin{align*}
        C_1 ={}& \left\|\exp\left(-\frac{|x|^{b+1-a}}{b+1-a}\right)\right\|_{L^p_c}^{-1} \\
        ={}& \left( V(\mathbb{S}^{N-1}) \int_0^\infty r^{-pc+N-1} \exp\left(-\frac{pr^{b+1-a}}{b+1-a}\right)\,dr \right)^{-1/p} \\
        ={}& \frac{p^\alpha V(\mathbb{S}^{N-1})^{-1/p}}{(b+1-a)^{\alpha-1/p}}\Gamma\left(\frac{pS_{p,a,b}}{b+1-a}\right)^{-1/p}.
    \end{align*}
    Then $\|v(k,\lambda)\|_{L^p_c} = |k|$. We can also compute its $L^p_a$-norm:
    \begin{align*}
        \|v(k,\lambda)\|_{L^p_a} ={}& |k|C_1\lambda^{-1/p}\left(V(\mathbb{S}^{N-1})\int_0^\infty r^{-pa+N-1}\exp\left(-\frac{pr^{b+1-a}}{b+1-a}\right)\,dr\right)^{1/p} \\
        ={}& |k|\lambda^{-1/p} \left(\frac{b+1-a}{p}\right)^{1/p}\left(\frac{pS_{p,a,b}}{b+1-a}\right)^{1/p}\\
        ={}& |k|\left(\frac{S_{p,a,b}}{\lambda}\right)^{1/p}.
    \end{align*}

    \item When $p=2$ and $(a,b)$ satisfy condition (2) in Theorem \ref{main p=2}, then $S_{2,a,b} = \frac{3b-a-N+3}{2}$.

    Functions in $\mathcal{M}_{2,a,b}$ can be written as
    \[ v(k,\lambda) = k C_1\lambda^\alpha |x|^{2b+2-N}\exp\left(-\frac{\lambda}{b+1-a}|x|^{b+1-a}\right),\]
    where
    \begin{align*}
        & \alpha = \frac{3b+3-a-N}{2(b+1-a)} = \frac{S_{2,a,b}}{b+1-a}, \\
        & C_1 = \left\| |x|^{2b+2-N} \exp\left(-\frac{|x|^{b+1-a}}{b+1-a}\right)\right\|_{L^2_c}^{-1}.
    \end{align*}
    Then $\|v(k,\lambda)\|_{L^2_c} = |k|$. Similarly, we can also get $\|v(k,\lambda)\|_{L^p_a} = |k|\left(\frac{S_{2,a,b}}{\lambda}\right)^{1/2}$.
\end{itemize}

It is easy to see that the norms of $v=v(k,\lambda)$ have the following relations (in both cases):
\begin{equation}\label{norm of v}
    \|v\|_{H^{p}_b} = \lambda \|v\|_{L^p_a},\quad \|v\|_{L^p_a} = \left(\frac{S_{p,a,b}}{\lambda}\right)^{1/p}\|v\|_{L^p_c}.
\end{equation}

Next, since functions in $\mathcal{M}_{p,a,b}$ minimize the following energy functional
\[ J(u) = \delta_{p,a,b}(u)\|u\|_{L^p_c}^p \ge 0,\]
we immediately get that $v(k,\lambda)\in\mathcal{M}_{p,a,b}$ satisfies the following equation:
\begin{equation}\label{eq}
        -\Div(|x|^{-pb}|\nabla v|^{p-2}\nabla v) + (p-1)\lambda^p|x|^{-pa}|v|^{p-2}v - pS_{p,a,b}\lambda^{p-1}|x|^{-pc}|v|^{p-2}v = 0.
\end{equation}

Last but not least, we give a statement about the attainability of the variational problem $\inf_{v\in\mathcal{M}_{p,a,b}}\|u-v\|_{H_c^p}$, which is an essential step to derive the stability results.

\vskip0.2in

\begin{proposition}\label{min}
    Let $N\ge 1$, $p\ge 2$. Suppose $(p,a,b)$ satisfy conditions in Theorem \ref{main} or Theorem \ref{main p=2}, then there exists $\delta_0>0$ such that for any $u\in H_b^p\cap L_a^p$ with $\delta_{p,a,b}(u)<\delta_0$, the following infimum is attainable:
    \begin{equation}\label{inf of u-v}
        \inf_{w\in\mathcal{M}_{p,a,b}}\|u-w\|_{L_c^p}.
    \end{equation}
\end{proposition}
\begin{proof}
    We first assume $(p,a,b)$ satisfy conditions in Theorem
    \ref{main} or condition (1) in Theorem \ref{main p=2}, then
    \[ pc = (p-1)a+b+1 = \frac{pbN}{N-p} \in [0, N), \]
    and so $|x|^{-c}\in L^p(\{|x|\le 1\})$. Picking a minimizing sequence $\{v_i=v(k_i,\lambda_i)\}_{i=1}^\infty$, we show $\{k_i\},\,\{\lambda_i\},\,\{1/\lambda_i\}$ are bounded. Since by \eqref{weak stability}
    \begin{equation*}
        \lim_{i\to \infty}\|u-v_i\|_{L^p_c} = \inf_{w\in\mathcal{M}}\|u-w\|_{L^p_c} \lesssim \delta_{p,a,b}(u)^{1/p}\|u\|_{L^p_c},
    \end{equation*}
    then
    \begin{equation}
        \delta_{p,a,b}(u)^{1/p}\|u\|_{L^p_c} \ge \lim_{i\to \infty} \|u-v(k_i,\lambda_i)\|_{L^p_c} \ge \limsup_{i\to \infty} |k_i| - \|u\|_{L^p_c}.
    \end{equation}
    Thus $\{k_i\}$ is bounded: $|k_i|\le K$, $\forall\, i$, and we can pick a subsequence such that $k_i\to k_0\in \R$. If $\lambda_i\to 0$, for any $x\in \R^N\backslash\{0\}$, we have
    \begin{equation*}
        |x|^{-c}|v_i(x)| = |k_i|C_1\lambda_i^{\alpha}|x|^{-c}\exp\left(-\lambda_i\frac{|x|^{b+1-a}}{b+1-a}\right) \le KC_1\lambda_i^{\alpha}|x|^{-c} \to 0,\quad i\to \infty.
    \end{equation*}
    If $\lambda_i\to +\infty$, for any $x\in\R^N\backslash\{0\}$, choosing an integer $m > \alpha$, we have
    \begin{align*}
        |x|^{-c}|v_i(x)| ={}& |k_i|C_1\lambda_i^{\alpha}|x|^{-c}\exp\left(-\lambda_i\frac{|x|^{b+1-a}}{b+1-a}\right)\\
        \le{}& KC_1\lambda_i^{\alpha}|x|^{-c} \cdot \frac{m!}{\left(\lambda_i \frac{|x|^{b+1-a}}{b+1-a}\right)^m}\\
        \lesssim{}& K\lambda_i^{-m + \alpha}|x|^{-m(b+1-a)-c} \to 0, \quad i\to \infty.
    \end{align*}
    Hence, no matter $\lambda_i\to 0$ or $1/\lambda_i \to 0$, we always have $|x|^{-c}v_i(x) \to 0$ pointwise for all $x\in\R^N\backslash\{0\}$. Since $u\in L^p_c$, there exist $\epsilon>0$ and $A>0$ s.t.
    \begin{equation*}
        \int_{\{|x|<\epsilon\}\cup \{|x|>A\}} |x|^{-pc}|u|^p\,dx > \frac 12 \|u\|_{L_c^p}^p.
    \end{equation*}
    In the region $\{\epsilon\le |x|\le A\}$, we have
    \begin{equation*}
        |x|^{-c}|v_i(x)| \lesssim \begin{cases}
            K\epsilon^{-c}\left(\sup \lambda_i\right)^\alpha, & \text{if $\lambda_i\to 0$},\\
            K\epsilon^{-m(b+1-a)-c}\left(\inf \lambda_i\right)^{-m+\alpha}, & \text{if $\lambda_i\to \infty$}.
        \end{cases}
    \end{equation*}
    Since the above dominated functions of $\{|x|^{-c}v_i\}$ are both in $L^p(\{\epsilon\le |x|\le A\})$, the Dominated Convergence Theorem implies
    \begin{align*}
        \delta_0\|u\|_{L^p_c}^p > \delta_{p,a,b}(u)\|u\|_{L^p_c}^p \gtrsim  \lim_{i\to \infty}\|u-v_i\|_{L^p_c}^p \ge{}& \limsup_{i\to \infty} \int_{\epsilon\le |x|\le A} |x|^{-pc}|u-v_i|^p\,dx \\
        ={}& \int_{\epsilon\le |x|\le A} |x|^{-pc}|u|^p\,dx > \frac 12 \|u\|^p_{L^p_c},
    \end{align*}
    which leads to a contradiction if $\delta_0\ll  1$. Thus, $\{\lambda_i\},\,\{1/\lambda_i\}$ are bounded, and we can pick a subsequence such that both $k_i\to k_0\in \R$ and $\lambda_i\to \lambda_0\in\R_+$, and there exist $t,T\in \R_+$ s.t. $0<t\le |\lambda_i|\le T<\infty$. Hence,
    \begin{align*}
        |x|^{-c}|v_i(x)| \lesssim KT^{\alpha}|x|^{-c}\cdot \chi_{\{|x|\le 1\}} + Kt^{-m+\alpha}|x|^{-m(b+1-a)-c}\cdot\chi_{\{|x|\ge 1\}} \in L^p,
    \end{align*}
    and so
    \[ \|u-v_i\|_{L^p_c} \to \|u-v(k_0,\lambda_0)\|_{L^p_c},\]
    showing that $v(k_0,\lambda_0)$ attains the infimum.

    Next, we assume $(p,a,b)$ satisfy condition (2) in Theorem \ref{main p=2}, then
    \[ 2b+2-N-c = \frac{3b-a+1}{2}-N = \frac{N}{4}(b+1-a)-\frac N2 < -\frac N2,\]
    and so $|x|^{2b+2-N-c}\in L^2(\{|x|\le 1\})$. Picking a minimizing sequence $\{v_i=v(k_i,\lambda_i)\}_{i=1}^{\infty}$, then
    \[ \limsup_i |k_i|-\|u\|_{L^2_c} \le \lim_i \|u-v_i\|_{L^2_c} \lesssim \delta_{2,a,b}(u)^{1/2}\|u\|_{L^2_c}\]
    and $\{k_i\}$ is bounded. Next, we show $\{\lambda_i\},\,\{1/\lambda_i\}$ are bounded. If $\lambda_i\to 0$, then
    \[ |x|^{-c}|v_i(x)| \le |k_i|C_1\lambda_i^\alpha |x|^{2b+2-N-c} \to 0,\quad \forall\, x\in \R^N\backslash\{0\},\]
    and
    \[ \delta_{2,a,b}(u)^{1/2}\|u\|_{L^2_c} \gtrsim \lim_i \|u-v_i\|_{L^2_c(\{\epsilon\le |x|\le M\})} = \|u\|_{L^2_c(\{\epsilon\le |x|\le M\})},\quad \forall\, 0<\epsilon<M. \]
    Letting $\epsilon \to 0$ and $M\to \infty$ will lead to a contradiction if $\delta_{2,a,b}(u)\ll 1$.

    If $\lambda_i\to +\infty$, then pick an integer $m>\alpha$,
    \begin{align*}
        |x|^{-c} |v_i(x)| \le{}& |k_i|C_1\lambda_i^\alpha |x|^{2b+2-N-c}\frac{m!}{\left(\lambda_i\frac{|x|^{b+1-a}}{(b+1-a)} \right)^m} \\
        \lesssim{}& K\lambda_i^{\alpha-m} |x|^{2b+2-N-c-m(b+1-a)} \to 0,\quad \forall\, x\in \R^N\backslash\{0\},
    \end{align*}
    and for all $0<\epsilon<M$,
    \begin{align*}
        \delta_{2,a,b}(u)^{1/2}\|u\|_{L^p_c} \gtrsim{}& \lim_i \|u-v_i\|_{L^2_c(\{\epsilon<|x|<M\})}\\
        ={}& \|u\|_{L^2_c(\{\epsilon<|x|<M \})}.
    \end{align*}
    Letting $\epsilon\to 0$ and $M\to \infty$ will lead to a contradiction if $\delta_{2,a,b}(u) \ll 1$. Thus, we may assume $k_i\to k_0\in\R$ and $\lambda_i\to \lambda_0\in\R_+$ with $0<t\le |\lambda_i|\le T<\infty$. Hence,
    \[ |x|^{-c}|v_i(x)| \lesssim KT^\alpha |x|^{2b+2-N-c}\chi_{\{|x|\le 1\}} + Kt^{\alpha-m} |x|^{2b+2-N-c-m(b+1-a)}\chi_{\{|x|>1\}}\in L^2.\]
    Similar to the above, it yields that $v(k_0,\lambda_0)$ attains the infimum.
\end{proof}

\subsection{On the second-order CKN inequalities}
Note that under the condition of Theorem \ref{main II - 2nd order}, $K_{p,a,b}=\frac{-b-(p-1)a-(p-1)(N-1)}{p}$. Functions in $\mathcal{M}_{p,a,b}^2$ can be written as
\[ v(k,\lambda) = k C_3 \lambda^\beta \int_{|x|}^\infty r^{1-N}\exp\left(-\frac{\lambda r^{b-a+1}}{b-a+1}\right)\,dr,\]
where
\[ \beta = \frac{-b-(p-1)a-(p-1)(N-1)}{(b-a+1)p} = \frac{K_{p,a,b}}{b-a+1},\]
and
\begin{align*}
    C_3 = \frac{p^\beta V(\mathbb{S}^{N-1})^{-1/p}}{(b-a+1)^{\beta-1/p}}\Gamma\left(\frac{pK_{p,a,b}}{b-a+1}\right)^{-1/p}.
\end{align*}
We can show by simple calculation that the norms of $v=v(k,\lambda)$ satisfy
\begin{equation}\label{norm of v 2nd}
    \|\Delta v\|_{L^{p}_b} = \lambda \|v\|_{H^p_a},\quad \|v\|_{H^p_a} = \left(\frac{K_{p,a,b}}{\lambda}\right)^{1/p}\|v\|_{H^p_c},\quad \|v\|_{H_c^p} = |k|.
\end{equation}
Also, $v\in \mathcal{M}_{p,a,b}^2$ satisfies the following equation:
\begin{equation}\label{eq2}
    \begin{aligned}
        & -\Delta(|x|^{-pb}|\Delta v|^{p-2}\Delta v) + (p-1)\lambda^p\Div(|x|^{-pa}|\nabla v|^{p-2}\nabla v)\\
        ={}& pK_{p,a,b}\lambda^{p-1}\Div(|x|^{-pc}|\nabla v|^{p-2}\nabla v).
    \end{aligned}
\end{equation}

\section{Strong-form stability results for CKN inequalities}\label{section 3}

For the first part of this section, we show the deficit $\delta_{p,a,b}(u)$ cannot control the $H^p_b$- or $L^p_a$-norms of $u$. For $u\in H^p_b\cap L^p_a$, define $\Phi_\lambda u(x) \coloneqq \lambda^{N/p-c}u(\lambda x)$, then
\begin{align*}
    & \|\Phi_\lambda u\|_{L^p_c} = \|u\|_{L^p_c}; \\
    & \|\Phi_\lambda u\|_{L^p_a} = \lambda^{a-c}\|u\|_{L^p_a}; \\
    & \|\Phi_\lambda u\|_{H^p_b} = \lambda^{b-c+1}\|u\|_{H^p_b},
\end{align*}
implying that
\[ \delta_{p,a,b}(\Phi_\lambda u) = \delta_{p,a,b}(u). \]
It is easy to see $\Phi_{\lambda^{-1}}(\Phi_\lambda u) = u$, and the set of minimizers is invariant under this transformation, i.e.,  $\Phi_\lambda \mathcal{M}_{p,a,b} = \mathcal{M}_{p,a,b}$, $\forall\, \lambda>0$. Thus,
\begin{align*}
    \inf_{v\in\mathcal{M}_{p,a,b}}\|\Phi_\lambda u-v\|_{H^p_b} = \inf_{v\in\mathcal{M}_{p,a,b}}\|\Phi_\lambda(u-v)\|_{H^p_b} = \lambda^{b-c+1} \inf_{v\in \mathcal{M}_{p,a,b}}\|u-v\|_{H^p_b}.
\end{align*}
\begin{proof}[Proof of Proposition \ref{prop1}]
    If there exist $C>0$ and $\alpha>0$ such that for all $u \in H^p_b \cap L^p_a \backslash \{0\}$,
    \begin{align*}
        \left( \frac{\inf_{v\in \mathcal{M}_{p,a,b}}\|u-v\|_{H^p_b} }{\|u\|_{L^p_c}} \right)^\alpha \le C \delta_{p,a,b}(u),
    \end{align*}
    replacing $u$ by $\Phi_\lambda u$ we get
    \begin{align*}
        \left(\frac{\lambda^{b-c+1}\inf_{v\in \mathcal{M}_{p,a,b}}\|u-v\|_{H^p_b}}{\|u\|_{L^p_c}}\right)^\alpha \le C \delta_{p,a,b}(\Phi_\lambda u) = C \delta_{p,a,b}(u),\quad \forall\,\lambda>0.
    \end{align*}
    Since $b+1\neq a\implies b-c+1\neq 0$, taking $\lambda\to \infty$ or $0$, we get a contradiction. Similarly, if
    \[ \left( \frac{\inf_{v\in \mathcal{M}_{p,a,b}}\|u-v\|_{L^p_a} }{\|u\|_{L^p_c}} \right)^\alpha \le C \delta_{p,a,b}(u),\]
    then
    \begin{align*}
        \left(\frac{\lambda^{a-c}\inf_{v\in \mathcal{M}_{p,a,b}}\|u-v\|_{L^p_a}}{\|u\|_{L^p_c}}\right)^\alpha \le C \delta_{p,a,b}(\Phi_\lambda u) = C \delta_{p,a,b}(u),\quad \forall\,\lambda>0,
    \end{align*}
    we also get a contradiction by letting $\lambda \to \infty$ or $0$.
\end{proof}
These properties of transformation indicate that, to establish the stability results of the strong-forms, it is only possible to strengthen the deficit $\delta_{p,a,b}(u)$ into $\tilde{\delta}_{p,a,b}(u)$, or reduce the sum of the $H^p_b$- and $L^p_a$-norms to their weighted product.

\subsection{The case $p=2$}

In this case, we are surprised to find that the minimizer of $\inf_{w\in \mathcal{M}_{2,a,b}}\|u-w\|_{L^2_c}$ satisfies some orthogonality conditions, and can help us derive the strong-form inequality directly.
\begin{proof}[Proof of Theorem \ref{plus version}: the case $p=2$]
    Write $S=S_{2,a,b}$ and $\tilde{\delta}_{2,a,b}(u)=\tilde{\delta}(u)$. Thanks to Proposition \ref{min}, there exists $0<\delta_0<1/2$ such that if $\tilde{\delta}(u)<2\delta_0$, then $\delta_{2,a,b}(u) < \delta_0$, and so
    \[ \inf_{w\in\mathcal{M}_{2,a,b}} \|u-w\|_{L_c^2} \]
    is attainable. If $\tilde{\delta}(u) \ge 2\delta_0$, then
    \begin{align*}
        \inf_{v\in \mathcal{M}_{2,a,b}}\frac{ \|u-v\|_{H_b^2}^2+\|u-v\|_{L^2_a}^2}{\|u\|_{H^2_b}^2+\|u\|_{L^2_a}^2} \le 1 \le \epsilon_1^{-1}\tilde{\delta}(u),
    \end{align*}
    and the statement holds.

    Next, we assume $\tilde{\delta}(u)<2\delta_0$. We can choose $v\in\mathcal{M}_{2,a,b}$ such that it attains $\inf_{w\in\mathcal{M}_{2,a,b}} \|u-w\|_{L_c^2}$. Assume $v(x) = v(k_0,\lambda_0)$, then by the minimality, we know $k_0\neq 0$ and
    \begin{align*}
        & \partial_{k}\big|_{k=k_0}\|u-v(k,\lambda)\|_{L^2_c} = 0,  \\
        & \partial_{\lambda}\big|_{\lambda=\lambda_0}\|u-v(k,\lambda)\|_{L^2_c} = 0.
    \end{align*}
    Since for $k\neq 0,\,\lambda>0$,
    \[  \partial_k v(k,\lambda) = \frac{v}{k},\quad \partial_\lambda v(k,\lambda) = \frac{\alpha}{\lambda}v-\frac{|x|^{b+1-a}}{b+1-a}v,\]
    we get
    \begin{align*}
        \int |x|^{-2c}(u-v)\frac{v}{k_0} = 0,\quad \int |x|^{-2c}(u-v)v\left(\frac{\alpha}{\lambda_0} - \frac{|x|^{b+1-a}}{b+1-a}\right) =0.
    \end{align*}
    Hence, the following orthogonality conditions hold:
    \begin{align}
        & \int |x|^{-(a+b+1)}(u-v)v\,dx = 0 \implies \|v\|_{L^2_c}^2 = \int |x|^{-(a+b+1)}uv\,dx , \label{or1}\\
        & \int |x|^{-2a} (u-v)v\,dx =0 \implies \|v\|_{L^2_a}^2 = \int |x|^{-2a}uv\,dx. \label{or2}
    \end{align}
    Since $v$ satisfies equation \eqref{eq}, testing $v$ by $u$, we obtain
    \begin{align*}
        & \int |x|^{-2b}\nabla v\cdot\nabla u + \lambda_0^2 \int |x|^{-2a}vu - 2S\lambda_0 \int |x|^{-(a+b+1)}vu = 0 \\
        \implies{}& \int |x|^{-2b}\nabla u\cdot\nabla v = -\lambda_0^2\|v\|_{L^2_a}^2 + 2S\lambda_0\|v\|_{L^2_c}^2.
    \end{align*}
    By \eqref{norm of v}, when $p=2$, $\|v\|_{H^2_b} = \lambda_0\|v\|_{L^2_a} = \sqrt{\lambda_0 S}\|v\|_{L^2_c}$, we get the third orthogonality condition:
    \begin{equation}\label{or3}
        \begin{aligned}
            & \int |x|^{-2b}\nabla u\cdot \nabla v \,dx = \|v\|_{H^2_b}^2 \\
        \implies{}& \int |x|^{-2b} (\nabla u-\nabla v)\cdot\nabla v\,dx=0.
        \end{aligned}
    \end{equation}
    By \eqref{or1}, \eqref{or2}, \eqref{or3}, we directly get
    \[ \|u-v\|^2 = \|u\|^2 - \|v\|^2 \]
    for the $H^2_b$-, $L^2_a$-,$L^2_c$-norms. And so
    \begin{equation}\label{p=2 middle}
        \begin{aligned}
            \|u-v\|_{H_b^2}^2 + \|u-v\|_{L^2_a}^2 ={}& (\|u\|_{H_b^2}^2-\|v\|_{H_b^2}^2) + (\|u\|_{L^2_a}^2-\|v\|_{L^2_a}^2) \\
            ={}& \|u\|_{H_b^2}^2 + \|u\|_{L_a^2}^2 - 2S^2 \|u\|_{L_c^2}^2 + 2S^2\|u-v\|_{L^2_c}^2 \\
        & - \big( \|v\|_{H_b^2}^2+\|v\|_{L_a^2}^2 - 2S^2\|v\|_{L^2_c}^2 \big)\\
            ={}& \tilde{\delta}(u)\|u\|_{L^2_c}^2 + 2S^2\|u-v\|_{L^2_c}^2 - \tilde{\delta}(v)\|v\|_{L^2_c}^2\\
        \le{}& \tilde{\delta}(u)\|u\|_{L^2_c}^2 + 2S^2\|u-v\|_  {L^2_c}^2.
        \end{aligned}
    \end{equation}
    By \eqref{weak stability}, the minimizer $v$ satisfies
    \[ \|u-v\|_{L^2_c}^2 \le C^{-1}(N,2,a,b)\|u\|_{L^2_c}^2\delta_{2,a,b}(u).\]
    Divided by $\|u\|_{L^2_c}^2$ on both sides of \eqref{p=2 middle}, we obtain
    \begin{align*}
        \frac{\|u-v\|_{H^2_b}^2+\|u-v\|_{L^2_a}^2}{\|u\|_{L^2_c}^2} \le{}& \tilde{\delta}(u) + C^{-1}(N,2,a,b)\delta_{2,a,b}(u) \lesssim \tilde{\delta}(u).
    \end{align*}
    Hence,
    \begin{align*}
        \frac{\|u-v\|_{H^2_b}^2+\|u-v\|_{L^2_a}^2}{\|u\|_{H^2_b}^2+\|u\|_{L^2_a}^2} \lesssim{}& \tilde{\delta}(u)\frac{\|u\|_{L^2_c}^2}{\|u\|_{H^2_b}^2+\|u\|_{L^2_a}^2} \lesssim \tilde{\delta}(u),
    \end{align*}
    and we prove \eqref{sum 2}.
\end{proof}
\begin{proof}[Proof of Theorem \ref{main p=2}]
    For $u\in H^2_b \cap L^2_a$, set $\lambda = \frac{\|u\|_{H^2_b}}{\|u\|_{L^2_a}}$ and let $\tilde{u}(x) = u(\lambda^{-\frac{1}{b+1-a}} x)$, then $\|\tilde{u}\|_{H^2_b} = \|\tilde{u}\|_{L^2_a}$, and $\tilde{\delta}_{2,a,b}(\tilde{u}) = 2\delta_{2,a,b}(u)$. For any $v\in \mathcal{M}_{2,a,b}$,
    \begin{align*}
            \frac{\|\tilde{u}-v\|_{H_b^2}^2 + \|\tilde{u}-v\|_{L^2_a}^2}{\|\tilde{u}\|_{H^2_b}^2 + \|\tilde{u}\|_{L^2_a}^2} ={}& \frac{\lambda^{-1}\|u-\tilde{v}\|_{H^2_b}^2 +\lambda \|u-\tilde{v}\|_{L^2_a}^2}{2\|\tilde{u}\|_{H^2_b}\|\tilde{u}\|_{L^2_a}} \\
            \ge{}& \frac{\|u-\tilde{v}\|_{H^2_b}\|u-\tilde{v}\|_{L^2_a}}{\|u\|_{H_b^2}\|u\|_{L^2_a}},
    \end{align*}
    where $\tilde{v}(x) = v(\lambda^{-\frac{1}{b+1-a}}x)$. By the symmetry of $\mathcal{M}_{p,a,b}$, we know $\tilde{v}\in\mathcal{M}_{2,a,b}$ and hence, applying \eqref{sum 2} to $\tilde{u}$ yields
    \[ \inf_{v\in\mathcal{M}_{2,a,b}} \frac{\|u-v\|_{H^2_b}\|u-v\|_{L^2_a}}{\|u\|_{H^2_b}\|u\|_{L^2_a}} \lesssim \tilde{\delta}_{2,a,b}(\tilde{u}) \lesssim \delta_{2,a,b}(u). \qedhere\]
\end{proof}
\begin{remark}\label{rmk 3.1}
    Theorem \ref{main p=2} is sharp, i.e.,
    \[ \inf\limits_{v\in \mathcal{M}_{2,a,b}}\frac{\|u-v\|_{H^2_b}\|u-v\|_{L^2_a}}{\|u\|_{H^2_b}\|u\|_{L^2_a}} \lesssim \delta_{2,a,b}(u)^t \]
    fails if $t>1$. In fact, choose $u(x) = v(1,1)(Ax)$ with $A=\diag(1,\dots,1,1+1/j)$, and then
    \[ \delta_{2,a,b}(u),\,\inf\limits_{v\in \mathcal{M}_{2,a,b}}\frac{\|u-v\|_{H^2_b}\|u-v\|_{L^2_a}}{\|u\|_{H^2_b}\|u\|_{L^2_a}}\sim j^{-2}. \]
\end{remark}

\subsection{The case $p>2$}
We first give some properties to simplify our problem.

Set $u(x) = u(r\theta) = \tilde{u}(r^l\theta)$ (this transform was used in \cite{Toshio1997, Lam2017}), where $r=|x|,\, \theta = \frac{x}{|x|}$ and $l=\frac{N-p-pb}{N-p}\in (0,1)$. Moreover, let $u_1(x) = \tilde{u}(l^{\frac{p-1}{p}}x)$, then
\begin{align*}
    \|u\|_{H^p_b}^p = \int |x|^{-pb}|\nabla u|^p\,dx ={}& \int_{\mathbb{S}^{N-1}}\int_0^\infty (|\nabla_ru|^2+r^{-2}|\nabla_\theta u|^2)^{p/2}r^{N-1-bp}\,drd\theta \\
    ={}& l^{p-1}\int_{\mathbb{S}^{N-1}}\int_0^\infty (|\nabla_r\tilde{u}|^2+l^{-2}r^{-2}|\nabla_\theta \tilde{u}|^2)^{p/2}\,drd\theta \\
    \ge{}& l^{p-1}\int_{\mathbb{S}^{N-1}}\int_0^\infty (|\nabla_r\tilde{u}|^2+r^{-2}|\nabla_\theta \tilde{u}|^2)^{p/2}\,drd\theta\\
    ={}& l^{p-1}\|\tilde{u}\|_{H_0^p}^p = l^{\frac{N(p-1)}{p}}\|u_1\|_{H^p_0}^p,
\end{align*}
where $\nabla_\theta$ is the gradient on the sphere $\mathbb{S}^{N-1}$. Also,
\begin{align*}
    \|u\|_{L^p_a}^p =\int |x|^{-pa}|u|^p\,dx ={}& \int_{\mathbb{S}^{N-1}}\int_0^\infty r^{-pa+N-1}|u|^p\,drd\theta\\
    ={}& l^{-1}\int_{\mathbb{S}^{N-1}}\int_0^\infty r^{\frac{p}{p-1}+N-1}|\tilde{u}|^p\,drd\theta\\
    ={}& l^{-1}\|\tilde{u}\|_{L^p_{-1/(p-1)}}^p = l^{\frac{N(p-1)}{p}}\|u_1\|_{L^p_{-1/(p-1)}}^p,\\
    \|u\|_{L^p_c}^p =\int |x|^{-pc}|u|^p\,dx ={}& \int_{\mathbb{S}^{N-1}}\int_0^\infty r^{-pc+N-1}|u|^p\,drd\theta\\
    ={}& l^{-1}\int_{\mathbb{S}^{N-1}}\int_0^\infty r^{\frac{p}{p-1}+N-1}|\tilde{u}|^p\,drd\theta\\
    ={}& l^{-1}\|\tilde{u}\|_{L^p_0}^p = l^{\frac{N(p-1)}{p}-1}\|u_1\|_{L^p_0}^p.
\end{align*}
Note that
\begin{align*}
    S_{p,a,b} ={}& \frac{N-1-(p-1)a-b}{p}\\ ={}& \frac{N-\frac{Npb}{N-p}}{p}\\
    ={}& \frac{N}{p}\cdot \frac{N-p-pb}{N-p} = lS(p,-\frac{1}{p-1},0),
\end{align*}
we get
\begin{equation}
    \begin{aligned}
        & \delta_{p,a,b}(u) \ge l\delta_{p,-\frac{1}{p-1},0}(u_1), \\
        & \tilde{\delta}_{p,a,b}(u) \ge l \tilde{\delta}_{p,-\frac{1}{p-1},0}(u_1).
    \end{aligned}
\end{equation}
Note that, under this transform $u\mapsto u_1$, the set $\mathcal{M}_{p,a,b}$ becomes $\mathcal{M}_{p,-\frac{1}{p-1},0}$, so it is enough to show Theorem \ref{plus version} and \ref{main} with $(a,b,c) = (-\frac{1}{p-1},0,0)$. In the following we will omit $a,b$ if there are no confusions, and write
\begin{equation*}
    \begin{aligned}
        &  H^p=H^p_0, \quad L^p=L^p_0, \quad S=S_{p,-\frac{1}{p-1},0} = \frac{N}{p},\\
        & \delta(u) = \delta_{p,-\frac{1}{p-1},0}(u), \quad \tilde{\delta}(u) =  \tilde{\delta}_{p,-\frac 1{p-1},0}(u),\quad \mathcal{M} = \mathcal{M}_{p,-\frac{1}{p-1},0}
    \end{aligned}
\end{equation*}
for convenience.

To prove Theorem \ref{plus version} and \ref{main}, we will use the following sharp inequalities on vectors proved in \cite{Figalli2022}.

\begin{lemma}[{\cite[Lemma 2.1 and 2.4]{Figalli2022}}]\label{primary}
    Let $y,z\in \R^n$ and $a,b\in \R$. Then, for any $\kappa>0$, there exist constants $c_0=c_0(p,\kappa)>0$ and $c_1=c_1(p,\kappa)>0$ such that the following statements hold:
    \begin{enumerate}[(i)]
        \item For $p>2$, \begin{align*}
            |y+z|^p \ge{}& |y|^p + p|y|^{p-2}y\cdot z \\
            & + \frac{1-\kappa}{2}\Big(p|y|^{p-2}|z|^2 + p(p-2)|w|^{p-2}\big(|y|-|y+z|\big)^2\Big)+c_0|z|^p,
        \end{align*}
        where
        \[ w=w(y,y+z) = \begin{cases}
            x, & \text{if}\ |y|\le |y+z|,\\
            (\frac{|y+z|}{|y|})^{\frac{1}{p-2}}(y+z), &\text{if}\ |y+z|\le |y|.
        \end{cases}\]
        \item For $p>2$, \begin{align*}
            |a+b|^p \le{}& |a|^p + p|a|^{p-2}ab + \left(\frac{p(p-1)}{2}+\kappa\right)|a|^{p-2}|b|^2 + c_1|b|^p.
        \end{align*}
    \end{enumerate}
\end{lemma}

Thanks to these primary inequalities, the proof for $p=2$ is valid as long as we can find some minimizer $v$ satisfying the following orthogonality properties:
\begin{equation}\label{con1}
    \begin{aligned}
        & \int |\nabla v|^{p-2}\nabla v\cdot (\nabla u-\nabla v)=0,\\
        & \int |x|^{\frac{p}{p-1}}|v|^{p-2}v(u-v) =0,\\
        & \int |v|^{p-2}v(u-v)=0.
    \end{aligned}
\end{equation}
We hope $v\in\mathcal{M}$ attaining $\inf_{w\in \mathcal{M}} \|u-w\|_{L^p}$ can give us the desired result, but such minimizer only satisfies
\begin{align*}
    \int |u-v|^{p-2}uv = 0 = \int |x|^{\frac{p}{p-1}}|u-v|^{p-2}uv,
\end{align*}
and it is difficult to show it satisfies \eqref{con1}. To settle it, we pick another element of $\mathcal{M}$.

\begin{definition}
    For $u\in H^p\cap L^p_{-1/(p-1)}$, set
    \begin{equation}\label{def of Pu}
        P_u = \left\{ v\in\mathcal{M} : \|v\|_{L^p}=1,\, \int |v|^{p-2}uv = \sup_{ \stackrel{w\in \mathcal{M}}{\|w\|_{L^p}=1} } \int |w|^{p-2}uw\right\}.
    \end{equation}
\end{definition}

\begin{proposition}\label{inf and max prop}
    There exists $\delta_1>0$ such that for any $u\in H^p\cap L^p_{-1/(p-1)}$ with $\delta(u)<\delta_1$, the following  supremum is attainable:
    \begin{equation}\label{inf and sup}
        \sup_{ \stackrel{w\in \mathcal{M}}{\|w\|_{L^p}=1} } \int |w|^{p-2}uw.
    \end{equation}
\end{proposition}
\begin{proof}
    Pick a minimizing sequence $\{v_i = v(k_i,\lambda_i)\}_{i=1}^\infty$, where $k_i\in \{1,-1\}$ and $\lambda_i>0$. Let
    \[ \mu \coloneqq \sup_{ \stackrel{w\in \mathcal{M}}{\|w\|_{L^p}=1} } \int |w|^{p-2}uw.\]
    If $\mu=0$, since at least one of $v(1,\lambda)$, $v(-1,\lambda)$ makes $\int |v|^{p-1}vu \ge 0$, it is easily to attain it. Next we assume $\mu>0$. Since $\||v_i|^{p-1}v_i\|_{L^{p'}}^{p'} = \|v_i\|_{L^p}^p = 1$, there exists $f\in L^{p'}$ s.t. $|v_i|^{p-2}v_i \rightharpoonup f$ weakly in $L^{p'}$. We recall the estimates of $|v_i(x)|$ in Proposition \ref{min}:
    \begin{equation*}
        |v_i(x)| \lesssim \begin{cases}
            K\left(\sup \lambda_i\right)^\alpha, & \text{If $\lambda_i\to 0$},\\
            K|x|^{-m(b+1-a)}\left(\inf \lambda_i\right)^{-m+\alpha}, & \text{If $\lambda_i\to \infty$}.
        \end{cases}
    \end{equation*}
    Hence, if $\lambda_i\to 0$ or $\lambda_i\to +\infty$, then for any $0<\epsilon<A<+\infty$ and any $\varphi \in C_0^\infty(\{\epsilon\le |x|\le A\})$,
    \begin{equation*}
        \int f\varphi = \lim_{i\to \infty} \int |v_i|^{p-2}v_i\varphi = 0 \implies f = 0 \,\,\text{for a.e. $\epsilon\le |x|\le A$},
    \end{equation*}
    yielding that $f=0$ a.e. and thus,
    \begin{equation*}
        0 = \int fu = \lim_{i\to \infty} \int |v_i|^{p-2}v_i u = \mu >0,
    \end{equation*}
    which is a contradiction. Thus, we can assume $k_i\equiv 1$ and $\lambda_i\to \lambda_0\in \R_+$, and so $v_i \to v_0\coloneqq v(1,\lambda_0)$ pointwise. We must have $f=|v_0|^{p-2}v_0$ a.e. and then
    \begin{equation*}
        \int |v_0|^{p-2}v_0u=\lim_{i\to \infty} \int |v_i|^{p-2}v_iu = \sup_{ \stackrel{w\in \mathcal{M}}{\|w\|_{L^p}=1} } \int |w|^{p-2}uw
    \end{equation*}
    as we desire.
\end{proof}
\begin{remark}
    Fix $p\ge 2$, for general $(a,b)$ satisfying the conditions in Theorem \ref{main}, if $\delta_{p,a,b}(u) < \frac{N-p-pb}{N-p}\delta_0$, using the transform $u(x) = \tilde{u}(r^l\theta)$ with $l=\frac{N-p-pb}{N-p}$, we know
    \[ \delta(\tilde{u}) \le \frac{N-p}{N-p-pb}\delta_{p,a,b}(u) <\delta_0.\]
    Thus,
    \[ \sup_{ \stackrel{w\in \mathcal{M}}{\|w\|_{L^p}=1} } \int |w|^{p-2}\tilde{u}w \]
    is attainable. As a result,
    \[ \sup_{ \stackrel{w\in \mathcal{M}_{p,a,b}}{\|w\|_{L^p_c}=1} } \int |x|^{-pc}|w|^{p-2}uw \]
    is attainable. So the statement of Proposition \ref{inf and max prop} holds for general $(a,b)$.
\end{remark}

\begin{proposition}\label{pro}
    There exist $\delta_2\le \min\{\delta_0,\delta_1\}$ and $C_2>0$ depending only on $N,p,a,b$ s.t. for any $u\in H^p\cap L^p_{-1/(p-1)}$ with $\delta(u) < \delta_2$, we have
    \begin{equation}
        \sup_{v\in P_u}\frac{\|u-\mu_v v\|_{L^p}}{\|u\|_{L^p}}\le C_2 \left(\inf_{w\in\mathcal{M}}\frac{\|u-w\|_{L^p}}{\|u\|_{L^p}}\right)^{1/2} \lesssim \delta(u)^{1/(2p)},
    \end{equation}
    where for $v\in P_u$, $\mu_v = \int |v|^{p-2}vu \ge 0$.
\end{proposition}
\begin{proof}
    For $u\in H^p\cap L^p_{-1/(p-1)}$ with $\delta(u)<\min\{\delta_0,\delta_1\}$, we can pick $v\in P_u$, $\mu = \int |v|^{p-2}vu$ and $w$ attaining $\inf_{w\in \mathcal{M}}\|u-w\|_{L^p}$. Since $\epsilon \coloneqq \|u-w\|_{L^p} \lesssim \delta(u)^{1/p} \ll 1$, we have
    \begin{equation}\label{est 7.2}
        \big|\|w\|_{L^p} - 1\big| \le \|w-u\|_{L^p} = \epsilon.
    \end{equation}
    From the fact that $v\in P_{u}$, we get
    \begin{align*}
        \|w\|_{L^p}^{p-1}\mu \ge{}& \int |w|^{p-2}wu\\
        ={}& \int |w|^{p-2}w(u-w) + \|w\|_{L^p}^p \\
        \ge{}& \|w\|_{L^p}^p - \|w\|_{L^p}^{p-1}\|u-w\|_{L^p},
    \end{align*}
    and so
    \begin{equation}\label{est 7.3}
        1\ge \int |v|^{p-2}vu = \mu \ge \|w\|_{L^p} - \epsilon.
    \end{equation}
    By \eqref{est 7.2},
    \begin{equation}\label{est 7.4}
        0\le 1-\mu \le \epsilon + 1-\|w\|_{L^p} \le 2\epsilon.
    \end{equation}
    By \eqref{est 7.2}, \eqref{est 7.4}, we immediately get
    \begin{equation}\label{est 7.5}
        \begin{split}
            & \left\|\frac{w}{\|w\|_{L^p}}-v\right\|_{L^p} \\
            \ge{}& \|u-\mu v\|_{L^p} - \|\mu v-v\|_{L^p} - \|u-w\|_{L^p} - \left\|\frac{w}{\|w\|_{L^p}}-w\right\|_{L^p}\\
            \ge{}& \|u-\mu v\|_{L^p} - |\mu-1| - \big|\|w\|_{L^p} - 1\big| - \epsilon\\
            \ge{}& \|u-\mu v\|_{L^p} - 4\epsilon.
        \end{split}
    \end{equation}
    Moreover,
    \begin{equation}\label{est 7.6}
        \begin{aligned}
            1 \ge \int |v|^{p-2}v\frac{w}{\|w\|_{L^p}} ={}& \int |v|^{p-2}v_i\left(\frac{w}{\|w\|_{L^p}}-u_i\right) + \mu\\
            \ge{}& \mu - \|u-w\|_{L^p} - \left\|\frac{w}{\|w\|_{L^p}}-w\right\|_{L^p}\\
            \ge{}& 1-4\epsilon > 0,
        \end{aligned}
    \end{equation}
    if $\epsilon<1/4$. Thus, $v$ and $w$ have the same sign, and we can assume $v=v(1,\lambda)$ and $w = v(\|w\|_{L^p},\rho)$. Now we have
    \begin{align*}
        \int |v|^{p-2}v\frac{w}{\|w\|_{L^p}} ={}& C_1^p \lambda^{\frac{N(p-1)^2}{p^2}}\rho^{\frac{N(p-1)}{p^2}}\int \exp\left(-((p-1)\lambda+\rho)\frac{|x|^{p/(p-1)}}{p/(p-1)}\right) \,dx\\
        ={}& C_1^p \left(\frac{\lambda^{\frac{p-1}{p}}\rho^{\frac 1p}}{(1-\frac 1p)\lambda + \frac 1p \rho}\right)^{\frac{N(p-1)}{p}} \int \exp\left(-(p-1)|x|^{p/(p-1)}\right)\,dx \\
        ={}& \left(\frac{\lambda^{\frac{p-1}{p}}\rho^{\frac 1p}}{(1-\frac 1p)\lambda + \frac 1p \rho}\right)^{\frac{N(p-1)}{p}} = \left(\frac{t^{1/p}}{1-1/p + t/p}\right)^{\frac{N(p-1)}{p}},
    \end{align*}
    where $t = \rho/\lambda$. By \eqref{est 7.6},
    \begin{equation}\label{est 7.7}
        0 \le 1 - \left(\frac{t^{1/p}}{1-1/p + t/p}\right)^{\frac{N(p-1)}{p}} \le 4\epsilon.
    \end{equation}
    It is easy to see $|t-1|\le 1/2$ as long as $\epsilon \ll 1$. Since
    \begin{equation*}
        \lim_{x\to 1} \frac{1-x^q}{1-x} = \lim_{x\to 1} \frac{-qx^{q-1}}{-1} = q>0,\quad \forall\, q\in \R_+,
    \end{equation*}
    \eqref{est 7.7} implies
    \begin{equation}\label{est 7.8}
        1- \frac{t^{1/p}}{1-1/p + t/p} \lesssim \epsilon.
    \end{equation}
    Finally, we calculate the distance of $\frac{w}{\|w\|_{L^p}}$ and $v$:
    \begin{align*}
        \left\|\frac{w}{\|w\|_{L^p}}-v\right\|_{L^p}^p ={}& C_1^p \int \left|\rho^{\frac{N(p-1)}{p^2}}e^{-\rho\frac{|x|^{p/(p-1)}}{p/(p-1)}}-\lambda^{\frac{N(p-1)}{p^2}}e^{-\lambda\frac{|x|^{p/(p-1)}}{p/(p-1)}}\right|^p\,dx \\
        ={}& C_1^p \int \left|t^{\frac{N(p-1)}{p^2}}e^{-t\frac{|x|^{p/(p-1)}}{p/(p-1)}}-e^{-\frac{|x|^{p/(p-1)}}{p/(p-1)}}\right|^p\,dx\\
        ={}& C_1^p \int_{\mathbb{S}^{N-1}}\int_0^\infty e^{-(p-1)r^{p/(p-1)}}r^{N-1} \left|1-t^{\frac{N(p-1)}{p^2}}e^{(1-t)\frac{r^{p/(p-1)}}{p/(p-1)}}\right|^p\,drd\theta\\
        ={}& C_1^p V(\mathbb{S}^{N-1}) \int_0^\infty e^{-(p-1)r^{p/(p-1)}}r^{N-1}|f_r(1)-f_r(t)|^p\,dx,
    \end{align*}
    where
    \begin{equation}
        f_r(s) = s^{\frac{N(p-1)}{p^2}}e^{(1-s)\frac{r^{p/(p-1)}}{p/(p-1)}}.
    \end{equation}
    For $|s-1|\le 1/2$, we have
    \begin{align*}
        f_r'(s) ={}& f_r(s) \left(\frac{N(p-1)}{p^2s} - \frac{r^{p/(p-1)}}{p/(p-1)}\right)\\
        \implies |f_r'(s)| \le{}& \left(\frac{2N(p-1)}{p^2} + \frac{r^{p/(p-1)}}{p/(p-1)}\right)|f_r(s)|\\
        \le{}& 2^{\frac{N(p-1)}{p^2}} \left(\frac{2N(p-1)}{p^2} + \frac{r^{p/(p-1)}}{p/(p-1)}\right)\exp\left(\frac 12 \frac{r^{p/(p-1)}}{p/(p-1)}\right),
    \end{align*}
    and so
    \begin{equation}
        |f_r(1)-f_r(t)|\le |1-t|\sup_{s\in[1,t]}|f_r'(s)| \lesssim |1-t|(1+r^{p/(p-1)})\exp\left(\frac 12 \frac{r^{p/(p-1)}}{p/(p-1)}\right).
    \end{equation}
    Hence, when $\epsilon \ll 1$, we have
    \begin{equation}\label{est 7.10}
        \begin{split}
            \left\|\frac{w}{\|w\|_{L^p}}-v\right\|_{L^p}^p \lesssim{}& |1-t|^p \int_0^\infty e^{-\frac{p-1}{2}r^{p/(p-1)}} r^{N-1} (1+r^{p/(p-1)})^p\,dr\\
            ={}& O(|1-t|^p).
        \end{split}
    \end{equation}
    Since
    \begin{equation*}
        \lim_{x\to 1}\frac{1-\frac 1p + \frac xp - x^{\frac 1p}}{(1-x)^2} = \lim_{x\to 1} \frac{\frac 1p\left(1 - x^{\frac 1p -1}\right)}{-2(1-x)} = \lim_{x\to 1}\frac{\frac 1p \left(1-\frac 1p\right) x^{\frac 1p-2}}{2} = \frac{1}{2p}\left(1-\frac 1p\right) >0,
    \end{equation*}
    \eqref{est 7.8} implies
    \begin{equation}\label{est 7.11}
        \epsilon \gtrsim 1-\frac{t^{1/p}}{1-1/p + t/p} = O(|1-t|^2) \implies |1-t|\lesssim \epsilon^{1/2}.
    \end{equation}
    Taking \eqref{est 7.10} and \eqref{est 7.11} into \eqref{est 7.5}, we get
    \begin{equation*}
        \begin{split}
            & |1-t| \gtrsim \left\|\frac{w}{\|w\|_{L^p}}-v\right\|_p \gtrsim \|u-\mu v\|_{L^p}-4\epsilon \\
            \implies{}& \|u-\mu v\|_{L^p} \lesssim 4\epsilon + |1-t| \lesssim \epsilon + \epsilon^{1/2} \lesssim \epsilon^{1/2},
        \end{split}
    \end{equation*}
    and we complete our proof.
\end{proof}

\begin{proof}[Proof of Theroem \ref{plus version}: the case $p>2$]
    We first prove it for $(a,b) = (-\frac{1}{p-1},0)$.
    If $\tilde{\delta}(u) \ge p\delta_2$, then
    \begin{equation*}
        \inf_{v\in \mathcal{M}} \frac{ \|u-v\|_{H^p}^p + (p-1)\|u-v\|_{L^p_{-1/(p-1)}}^p }{\|u\|_{H^p}^p+(p-1)\|u\|_{L^p_{-1/(p-1)}}^p} \le 1 \lesssim \tilde{\delta}(u)^{1/p},
    \end{equation*}
    and we are done. Next, assume $\tilde{\delta}(u)<p\delta_2$, then $\delta(u) < \delta_2$, and $P_u\neq\varnothing$. Thus, we can pick $v\in P_u$, $\mu = \int |v|^{p-2}vu\ge 0$. Since $\|v\|_{L^p}=1$, we may assume $v=v(1,\lambda)$. By the maximality of $v$,
    \begin{align*}
        0 ={}& \frac{d}{d\lambda} \int |v(1,\lambda)|^{p-2}v(1,\lambda)u\\
        ={}& (p-1) \int |v|^{p-2}u\partial_\lambda v \\
        ={}& \frac{N(p-1)^2}{p^2\lambda}\int |v|^{p-2}uv - \frac{(p-1)^2}{p}\int |x|^{\frac{p}{p-1}} |v|^{p-2} uv.
    \end{align*}
    Thus,
    \begin{equation}\label{a norm}
        \int |x|^{\frac{p}{p-1}}|v|^{p-2}uv = \frac{p}{(p-1)^2}\cdot \frac{N(p-1)^2}{p^2\lambda}\mu = \frac{N\mu}{p\lambda} = \frac{S\mu}{\lambda}.
    \end{equation}
    By \eqref{eq}, for any $v\in \mathcal{M}_{p,a,b}$ and $u\in H_b^p\cap L_a^p$, we have
    \begin{equation}\label{weak equation of minimizer}
        \int |x|^{-pb}|\nabla v|^{p-2}\nabla v\cdot \nabla u + (p-1)\lambda^p\int |x|^{-pa}|v|^{p-2}uv - Sp\lambda^{p-1}\int |x|^{-pc}|v|^{p-2}uv = 0.
    \end{equation}
    By \eqref{weak equation of minimizer},
    \begin{equation}\label{b norm}
        \begin{split}
            \int |\nabla v|^{p-2}\nabla v \cdot \nabla u ={}& Sp\lambda^{p-1}\int |v|^{p-2}uv - (p-1)\lambda^p\int|x|^{\frac{p}{p-1}}|v|^{p-2}uv \\
            ={}& Sp\lambda^{p-1}\mu - (p-1)\lambda^p \frac{S\mu}{\lambda} = S\lambda^{p-1}\mu.
        \end{split}
    \end{equation}
    Note that
    \begin{equation*}
        \|v\|_{L^p_{-1/(p-1)}} = (S/\lambda)^{1/p}, \quad \|\nabla v\|_{L^p} = \lambda(S/\lambda)^{1/p}.
    \end{equation*}
    Write $\bar v = \mu v$, then the following orthogonality conditions hold:
    \begin{equation}\label{orth-111}
        \begin{aligned}
            & \int |\bar v|^{p-2}\bar v(u-\bar v) = \mu^{p-1}\int |v|^{p-2}vu - \mu^p \|v\|_{L^p}^p = 0, \\
            & \int |x|^{\frac{p}{p-1}}|\bar v|^{p-2}\bar{v}(u-\bar v) = \mu^{p-1}\int |x|^{\frac{p}{p-1}}|v|^{p-2}vu - \mu^p \|v\|_{L^p_{-1/(p-1)}}^p = 0,\\
            & \int |\nabla \bar v|^{p-2}\nabla\bar v\cdot(\nabla u-\nabla \bar v) = \mu^{p-1}\int |\nabla v|^{p-2}\nabla v\cdot\nabla u - \mu^p \|\nabla v\|_{L^p}^p = 0.
        \end{aligned}
    \end{equation}
    Applying \eqref{primary} with $(y,z)=(\nabla \bar v,\nabla u-\nabla \bar v)$ and integrate over $\R^N$, we find that
    \begin{equation}\label{111}
        \begin{aligned}
            \|\nabla u\|_{L^p}^p \ge{}& \|\nabla \bar v\|_{L^p}^p + p \int |\nabla \bar v|^{p-2}\nabla \bar v\cdot(\nabla u-\nabla \bar v) \\
            & + \frac{1-\kappa}{2}\bigg(\int p|\nabla \bar v|^{p-2}|\nabla u-\nabla \bar v|^2 \\
            & \qquad + p(p-2) \int |w(\nabla\bar v,\nabla u)|^{p-2}(|\nabla\bar v|-|\nabla u|)^2 \bigg) \\
            & + c_0 \|\nabla u-\nabla\bar v\|_{L^p}^p \\
            \ge{}& \|\nabla\bar v\|_{L^p}^p + c_0 \|\nabla u-\nabla\bar v\|_{L^p}^p.
        \end{aligned}
    \end{equation}
    Similarly, taking $(y,z)=(|x|^{\frac{1}{p-1}}\bar v, |x|^{\frac{1}{p-1}}(u-\bar v))$ and integrating over $\R^N$, we obtain
    \begin{equation}\label{222}
        \|u\|_{L^p_{-1/(p-1)}}^p \ge \|\bar v\|_{L^p_{-1/(p-1)}}^p + c_0 \|u-\bar v\|_{L^p_{-1/(p-1)}}^p.
    \end{equation}
    Lastly, taking $(a,b)=(\bar v,u-\bar v)$ and integrating over $\R^N$, then use H\"{o}lder's inequality, it holds
    \begin{equation}\label{333}
        \begin{aligned}
            \|u\|_{L^p}^p \le{}& \|\bar v\|_{L^p}^p + p \int |\bar v|^{p-2}\bar v(u-\bar v) \\
            & + \left(\frac{p(p-1)}{2} + \kappa \right)\int |\bar v|^{p-2}|u-\bar v|^2 + c_1\|u-\bar v\|_{L^p}^p \\
            \le{}& \|\bar v\|_{L^p}^p + \left(\frac{p(p-1)}{2} + \kappa \right)\|\bar v\|_{L^p}^{p-2}\|u-\bar v\|_{L^p}^2+c_1\|u-\bar v\|_{L^p}^p.
        \end{aligned}
    \end{equation}
  Now  \eqref{111}, \eqref{222} and \eqref{333}  imply  that
    \begin{align*}
        & \|\nabla u\|_{L^p}^p + (p-1)\|u\|_{L^p_{-1/(p-1)}}^p -Sp\|u\|_{L^p}^p \\
        \ge{}& \|\nabla \bar v\|_{L^p}^p + (p-1)\|\bar v\|_{L^p_{-1/(p-1)}}^p - Sp\|\bar v\|_{L^p}^p \\
        & + c_0 \Big( \|\nabla u-\nabla \bar v\|_{L^p}^p +(p-1)\|u-\bar v\|_{L^p_{-1/(p-1)}}^p \Big) \\
        & - \left(\frac{p(p-1)}{2} + \kappa \right)\|\bar v\|_{L^p}^{p-2}\|u-\bar v\|_{L^p}^2 - c_1\|u-\bar v\|_{L^p}^p.
    \end{align*}
    Since
    \[ \|\bar v\|_{L^p}^p = \int |\bar v|^{p-2}\bar v u \le \|\bar v\|_{L^p}^{p-1}\|u\|_{L^p} \implies \|\bar v\|_{L^p} \le \|u\|_{L^p}, \]
    and by Proposition \ref{pro},
    \[ \|u-\bar v\|_{L^p} \lesssim \delta(u)^{1/(2p)}\|u\|_{L^p},\]
    we finally obtain
    \begin{align*}
        \frac{\|\nabla u-\nabla \bar v\|_{L^p}^p +(p-1)\|u-\bar v\|_{L^p_{-1/(p-1)}}^p}{\|\nabla u\|_{L^p}^p + (p-1) \|u\|_{L^p_{-1/(p-1)}}^p} \le{}& \frac{\|\nabla u-\nabla \bar v\|_{L^p}^p +(p-1)\|u-\bar v\|_{L^p_{-1/(p-1)}}^p}{Sp\|u\|_{L^p}^p}
        \\
        \lesssim{}& \tilde{\delta}(u) - \tilde{\delta}(\bar v) + \delta(u)^{1/p} + \delta(u)^{1/2}\\
        \le{}& \tilde{\delta}(u) + \delta(u)^{1/p} + \delta(u)^{1/2}
        \\
        \lesssim{}& \tilde{\delta}(u)^{1/p},
    \end{align*}
    and we get \eqref{sum 1} with $(a,b,c) = (-\frac{1}{p-1},0,0)$.

    By our former discussions, for general $(a,b)$ it holds
    \[ \tilde{\delta}_{p,a,b}(u) \gtrsim \tilde{\delta}_{p,-1/(p-1),0}(u_1),\]and \eqref{sum 1} can be deduced directly from the above result.
\end{proof}

Finally, similar to the proof of Theorem \ref{main p=2}, we can deduce Theorem \ref{main} directly from Theorem \ref{plus version}:
\begin{proof}[Proof of Theorem \ref{main}]
    Set $\lambda = \frac{\|u\|_{H^p_b}}{\|u\|_{L^p_a}}$ and let $\Tilde{u}(x) = u(\lambda^{-\frac{1}{b+1-a}} x)$, then $\|\Tilde{u}\|_{H^p_b} = \|\tilde{u}\|_{L^p_a}$, and $\tilde{\delta}_{p,a,b}(\tilde{u}) = p\delta_{p,a,b}(u)$. For any $v\in \mathcal{M}_{p,a,b}$,
    \begin{align*}
        \frac{\|\tilde{u}-v\|_{H_b^p}^p + (p-1)\|\tilde{u}-v\|_{L^p_a}^p}{\|\tilde{u}\|_{H^p_b}^p + (p-1)\|\tilde{u}\|_{L^p_a}^p} ={}& \frac{\lambda^{-p+1}\|u-\tilde{v}\|_{H^p_b}^p + (p-1)\lambda \|u-\tilde{v}\|_{L^p_a}^p}{p\|\tilde{u}\|_{H^p_b}\|\tilde{u}\|_{L^p_a}^{p-1}} \\
        \ge{}& \frac{\|u-\tilde{v}\|_{H^p_b}\|u-\tilde{v}\|_{L^p_a}^{p-1}}{\|u\|_{H_b^p}\|u\|_{L^p_a}^{p-1}},
    \end{align*}
    where $\tilde{v}(x) = v(\lambda^{-\frac{1}{b+1-a}}x)$. By the symmetry of $\mathcal{M}_{p,a,b}$, we know $\tilde{v}\in\mathcal{M}_{p,a,b}$ and hence, applying \eqref{sum 1} to $\tilde{u}$ we obtain
    \[ \inf_{v\in\mathcal{M}_{p,a,b}} \frac{\|u-v\|_{H^p_b}\|u-v\|_{L^p_a}^{p-1}}{\|u\|_{H^p_b}\|u\|_{L^p_a}^{p-1}} \lesssim \tilde{\delta}_{p,a,b}(\tilde{u})^{1/p} \lesssim \delta_{p,a,b}(u)^{1/p}. \qedhere\]
\end{proof}

\section{Stabilities for the second-order CKN inequalities}
To prove the weak stability result, we first state a weighted $L^p$-Poincar\'{e} inequality, which was used in Do \cite{Do2023}:
\begin{lemma}[{\cite[Corollary 4.1]{Do2023}}]\label{Lemma Poincare}
    For some $m>0$, $N-p>\mu \ge 0$, $\gamma \ge \frac{N-p-\mu}{N-p}$ and $\tilde{\lambda}>0$, we have for $v\in C_0^\infty(\R^N\backslash \{0\})$ that
    \[ \tilde{\lambda}^{\frac{N\mu}{N-p}-p-\mu} \int_{\R^N}\frac{|\nabla v(y)|^p}{|y|^\mu}e^{-m|\tilde{\lambda} y|^\gamma}\,dy \ge C(N,p,\gamma,m,\mu)\inf_c \int_{\R^N} \frac{|v(y)-c|^p}{|y|^{\frac{N\mu}{N-p}}}e^{-m|\tilde{\lambda} y|^\gamma}\,dy.\]
\end{lemma}

\begin{proof}[Proof of Theorem \ref{main I - 2nd order}]
    For $p \ge 2$, there exists $m_p\in(0,1)$ such that
    \[ \mathcal{R}_p(s,t) \ge m_p|s-t|^p.\]
    From \eqref{identity-I of 2nd order}, when $pb+(p-1)N<0$ (i.e., $\epsilon=-1$), it holds that
    \begin{equation*}
        \begin{aligned}
            & \|\Delta u\|_{L_b^p}\|u\|_{H_a^p}^{p-1} - K_{p,a,b}\|u\|_{H_c^p}^p \\
           &   \ge{} \frac{m_p}{p}\int \frac{1}{|x|^{pb}}\left|\lambda^{1/p}|x|^{b-a-1}x\cdot\nabla u + \lambda^{-\frac{p-1}{p}\Delta u}\right|^p\,dx \\
             & ={} \frac{m_p}{p}\int \frac{\lambda^{1-p}}{|x|^{pb}} e^{-\frac{p\lambda|x|^{b-a+1}}{b-a+1}}\left|\nabla \cdot\left(e^{\frac{\lambda|x|^{b-a+1}}{b-a+1}}\nabla u\right)\right|^p\,dx.
        \end{aligned}
    \end{equation*}
    When $u$ is radial, set $w(r) = r^{-1}e^{\frac{\lambda r^{b-a+1}}{b-a+1}}u'(r)$, the divergence term becomes
    \[ \nabla\cdot(w(r)x) = rw'(r) + Nw(r) = r^{1-N}(r^Nw(r))'. \]
    Thus, we obtain
    \begin{equation*}
        \begin{aligned}
            &\|\Delta u\|_{L_b^p}\|u\|_{H_a^p}^{p-1} - K_{p,a,b}\|u\|_{H_c^p}^p \\
             & \ge{} \frac{m_p}{p}\int\frac{\lambda^{1-p}}{|x|^{pb+p(N-1)}}e^{-\frac{p\lambda|x|^{b-a+1}}{b-a+1}}\Big|\nabla(|x|^Nw(|x|))\Big|^p\,dx.
        \end{aligned}
    \end{equation*}
    Hence, by setting $\mu = pb+p(N-1)$, $\gamma = b-a+1$, $m = \frac{p}{b-a+1}$ and $\tilde{\lambda} = \lambda^{\frac{1}{b-a+1}}$ in Lemma \ref{Lemma Poincare}, it follows that
    \begin{equation*}
        \begin{aligned}
            &\|\Delta u\|_{L_b^p}\|u\|_{H_a^p}^{p-1} - K_{p,a,b}\|u\|_{H_c^p}^p \\
            \ge{}& C(N,p,a,b)\inf_c \int \frac{1}{|x|^{\frac{N\mu}{N-p}}}e^{-\frac{p\lambda|x|^{b-a+1}}{b-a+1}}\Big||x|^Nw(|x|)-c\Big|^p\,dx \\
            ={}& C(N,p,a,b)\inf_c \int \frac{1}{|x|^{\frac{N\mu}{N-p}+(N-1)p}}\Big|\nabla u(x)- c|x|^{1-N}e^{-\frac{\lambda|x|^{b-a+1}}{b-a+1}}\frac{x}{|x|}\Big|^p\,dx \\
            ={}& C(N,p,a,b)\inf_c \int \frac{1}{|x|^{\frac{N\mu}{N-p}+(N-1)p}}|\nabla u - \nabla v(c,\lambda)|^p\,dx.
        \end{aligned}
    \end{equation*}
    By our conditons for $(p,a,b)$, we have $\frac{N\mu}{N-p}+(N-1)p = pc$, and the proof is done.
\end{proof}
\begin{remark}
    The conditions for $(p,a,b)$ in Theorem \ref{main I - 2nd order} are set such that $(\mu,\gamma,m,\tilde{\lambda})$ meet the requirement of Lemma \ref{Lemma Poincare}.
\end{remark}

The process from weak stability to strong is similar to Section 3. Let $Q_u$ be the set of $v\in\mathcal{M}_{p,a,b}^2$ attaining
\[ \sup_{\stackrel{w\in \mathcal{M}^2_{p,a,b}}{\|w\|_{H_c^p}=1}} \int |\nabla w|^{p-2}\nabla w\cdot \nabla u. \]
Since
\[ \nabla v(k,\lambda) = -C_3k\lambda^\beta |x|^{1-N}e^{-\frac{\lambda|x|^{b-a+1}}{b-a+1}}\frac{x}{|x|}\]
and $\beta=\frac{K_{p,a,b}}{b-a+1}>0$, using the methods in Proposition \ref{min} and \ref{inf and max prop}, we can show
\begin{proposition}
    There exists $\sigma_0>0$ such that for any $u\in H_{b,a}^{p,p}$ with $\sigma_{p,a,b}(u)<\sigma_0$, the following  infimum and supremum are attainable:
    \begin{equation}
        \inf_{w\in \mathcal{M}_{p,a,b}^2} \|u-w\|_{H^p_c}, \quad \sup_{ \stackrel{w\in \mathcal{M}_{p,a,b}^2}{\|w\|_{H^p_c}=1} } \int |\nabla w|^{p-2}\nabla u\cdot \nabla w.
    \end{equation}
\end{proposition}

Next, we can check each step in Proposition \ref{pro} and obtain
\begin{proposition}
    There exist $\sigma_1\le \sigma_0$ and $C_4>0$ depending only on $N,p,a,b$ s.t. for any $u\in H^{p,p}_{b,a}$ with $\sigma_{p,a,b}(u) < \sigma_1$, we have
    \begin{equation}
        \sup_{v\in Q_u}\frac{\|u-\mu_v v\|_{H^p_c}}{\|u\|_{H^p_c}}\le C_4 \left(\inf_{w\in\mathcal{M}_{p,a,b}^2}\frac{\|u-w\|_{H^p_c}}{\|u\|_{H^p_c}}\right)^{1/2} \lesssim \sigma_{p,a,b}(u)^{1/(2p)},
    \end{equation}
    where for $v\in Q_u$, $\mu_v = \int |\nabla v|^{p-2}\nabla v\cdot\nabla u \ge 0$.
\end{proposition}

By \eqref{eq2} and the maximality property of $Q_u$, for $v\in Q_u$, $\bar v \coloneqq \mu_v v$ satisfies the following orthogonality conditions:
\begin{equation}\label{2nd orth-111}
    \begin{aligned}
        & \int \frac{1}{|x|^{pc}}|\nabla \bar v|^{p-2}\nabla \bar v\cdot (\nabla u-\nabla \bar v) = 0, \\
        & \int \frac{1}{|x|^{pa}} |\nabla \bar v|^{p-2}\nabla \bar{v}\cdot (\nabla u-\nabla \bar v) = 0,\\
        & \int \frac{1}{|x|^{pb}}|\Delta \bar v|^{p-2}\Delta\bar v(\Delta u-\Delta \bar v) = 0.
    \end{aligned}
\end{equation}
\begin{proof}[Proof of Theorem \ref{main II - 2nd order}]
    When $p=2$, pick $v\in \mathcal{M}_{2,a,b}^2$ that attains $\inf_{w\in \mathcal{M}_{2,a,b}^2}\|u-w\|_{H_c^p}$, we can show $v$ satisfy \eqref{2nd orth-111} and the proof is similar to that of \eqref{result}.

    When $p>2$, pick $v\in Q_u$ and let $\mu \coloneqq \mu_v = \int |\nabla v|^{p-2}\nabla v\cdot \nabla u$. By Lemma \ref{primary},
    \begin{align*}
        \|\Delta u\|_{L_b^p}^p \ge{}& \|\Delta \bar v \|_{L_b^p}^p + c_0\|\Delta u-\Delta \bar v\|_{L_b^p}^p, \\
        \|u\|_{H_a^p}^p \ge{}& \|\bar v\|_{H_a^p}^p + c_0\|u-\bar v\|_{H_a^p}^p,\\
        \|u\|_{H_c^p}^p \le{}& \|\bar v\|_{H_c^p}^p + \left(\frac{p(p-1)}{2}+ \kappa\right)\int |\nabla \bar v|^{p-2}|\nabla \bar v-\nabla u|^2 + c_1\|u-\bar v\|_{H_c^p}^p.
    \end{align*}
    As a result,
    \begin{equation}\label{4.4}
        \begin{aligned}
            & \| \Delta u\|_{L_b^p}^p + (p-1)\|u\|_{H_a^p}^p
        -pK_{p,a,b}\|u\|_{H_c^p}^p\\
            \ge{}& \|\Delta \bar v\|_{L_b^p}^p + (p-1)\|\bar v\|_{H_a^p}^p-pK_{p,a,b}\|\bar v\|_{H_c^p}^p\\
            & + c_0\Big( \|\Delta u-\Delta \bar v\|_{L_b^p}^p + (p-1)\|u-\bar v\|_{H_a^p}^p \Big)\\
            & -\left(\frac{p(p-1)}{2}+\kappa\right)\|\bar v\|_{H_c^p}^{p-2}\|u-\bar v\|_{H_c^p}^2 - c_1\|u-\bar v\|_{H_c^p}^p.
        \end{aligned}
    \end{equation}
    Define
    \[ \tilde{\sigma}(u)\coloneqq \frac{\|\Delta u\|_{L_b^p}^p + (p-1)\|u\|_{H_a^p}^p}{\|u\|_{H_c^p}^p} - pK_{p,a,b},\]
    then $\tilde{\sigma}(u)\ge p\sigma_{p,a,b}(u)\ge 0$, and \eqref{4.4} indicates
    \begin{equation*}
        \frac{\|\Delta u-\Delta \bar v\|_{L_b^p}^p + (p-1)\|u-\bar v\|_{H_a^p}^p}{\|u\|_{H_c^p}^p} \lesssim \tilde{\sigma}(u) + \sigma_{p,a,b}(u)^{1/2}+\sigma_{p,a,b}(u)^{1/p}\lesssim \tilde{\sigma}(u)^{1/p}.
    \end{equation*}
    Hence, the proof is done by using the transformation as in the proof of Theorem \ref{main}.
\end{proof}
\begin{remark}
    When $u$ is not radial, it is hard to get the weak stability as in Theorem \ref{main I - 2nd order}, but once it is done, we can deduce a strong-form result by using our methods.
\end{remark}



\end{document}